\documentclass[final,leqno,onefignum,onetabnum]{article}

\usepackage[T1]{fontenc}
\usepackage{tgtermes}
\usepackage{textcomp}
\usepackage{amsmath,amsthm} 

\usepackage{epsfig}
\usepackage{bbm}
\usepackage{ctable}
\usepackage{multirow}
\usepackage{xspace}
\usepackage{epsfig}
\usepackage{graphicx}
\usepackage[labelformat = simple, caption = false]{subfig}

\newcommand{\mathscr}{\mathcal}

\newtheorem{theorem}{Theorem}[section]
\newtheorem{lemma}[theorem]{Lemma}
\newtheorem{corollary}[theorem]{Corollary}
\newtheorem{assumption}[theorem]{Assumption}
\newtheorem{example}[theorem]{Example}

\DeclareMathOperator*{\argmin}{arg\,min}
\DeclareMathOperator{\prox}{prox}

\newcommand{\Rb}{\mathbbm{R}}

\newcommand{\Gc}{\mathcal{G}}

\newcommand\mbR{\mbox{$\mathbbm{R}$}}

\newcommand\mcK{\mathcal{K}}

\DeclareMathOperator*{\argmax}{arg\,max}

\newenvironment{tightlist}[1]{%
    \list{{\textup{(\roman{enumi})}}}{\settowidth\labelwidth{{\textup{(#1)}}}
    \leftmargin -6pt \advance\leftmargin\labelsep \itemindent \parindent
    \parsep 0pt plus 1pt minus 1pt \topsep 0pt \itemsep 0pt
    \usecounter{enumi}}}{\endlist}

\title{Selective Linearization For Multi-Block Convex Optimization}
\author{
 Yu Du\thanks{
Rutgers University, Department of Management Science and Information Systems, Piscataway, NJ 08854, USA, Email:{duyu@rutgers.edu}}
\and
 Xiaodong Lin\thanks{
Rutgers University, Department of Management Science and Information Systems, Piscataway, NJ 08854, USA, Email:{linxd@rci.rutgers.edu}}
\and
 Andrzej Ruszczy\'nski\thanks{
Rutgers University, Department of Management Science and Information Systems, Piscataway, NJ 08854, USA, Email:{rusz@rutgers.edu}}
}
 
\date{November 6, 2015; revised: March 3, 2016; August 12, 2016}

\begin{document}
\maketitle


\begin{abstract}
We consider the problem of minimizing a sum of several convex non-smooth functions. We introduce a new algorithm called the selective linearization method, which iteratively linearizes all but one of the functions and employs simple proximal steps. The algorithm is a form of multiple operator splitting in which the order of processing partial functions is not fixed, but rather determined in the course of
calculations.
Global convergence is proved and estimates of the convergence rate are derived. Specifically, the number of iterations needed to achieve solution accuracy $\varepsilon$ is of order $\mathcal{O}\big(\ln(1/\varepsilon)/\varepsilon\big)$. We also illustrate the operation of the algorithm on structured regularization problems.
\end{abstract}



\section{Introduction}
 In recent years, we have seen extensive development of the theory and methods for \emph{structured regularization}, one of the most fundamental techniques to address the ``big data'' challenge. The basic problem is to minimize the following objective function with two components (blocks):
\begin{equation}
\label{problem1}
\min \Big[ {F}(x)=f_1(x)+f_2(x) \Big],
\end{equation}
where $f_1(\cdot)$ is the loss function and $f_2(\cdot)$ is a penalty function that imposes
{structured regularization} to the model. Both functions are usually convex, but may be nonsmooth. Many data mining and machine learning problems can be cast within this framework, and efficient methods were proposed for these problems. The first group are the
\emph{operator splitting} methods originating from \cite{Douglas-Rachford} and \cite{Peaceman-Rachford}, and later developed and analyzed by \cite{Bauschke-Combettes,Combettes2009,Eckstein-Bertsekas,Lions-Mercier}, among others.
Their dual versions, known as \emph{Alternating Direction Methods of Multipliers} (ADMM) (see,
 \cite{gabay1976dual,glowinski1975approximation,glowinski1989augmented}), found many applications
 in signal processing (see, \emph{e.g.}, \cite{Boyd-et-al,Combettes-Pesquet}, and the references therein). Sometimes, they are called \emph{split Bregman methods}  (see, \emph{e.g.}, \cite{Goldstein-Osher,Ye-Xie}).

The \emph{Alternating Linearization Method} (ALIN) of \cite{kiwiel1999proximal} handles problems of form \eqref{problem1} by introducing an additional improvement test to the operator splitting methods, which decides whether the proximal center should be updated or stay unchanged, and which of the operator splitting formulas should be applied at the current iteration.
Its convergence mechanism is different than that of the splitting methods; it adapts some ideas of bundle methods of nonsmooth optimization \cite{HUL-book,Kiwiel-book,ruszczynski2006nonlinear}. The recent application of ALIN to structured regularization problems in \cite{lin2014alternating} proved very successful, with fast convergence, good accuracy, and scalability to very large dimensions.

Most of existing techniques for structured regularization are designed to handle the two-block problem of form \eqref{problem1}.

In this paper, we plan to extend the {ALIN} framework to optimization problems involving multiple components. Namely, we aim to solve the following problem:
\begin{equation}
\label{problem}
\min \Big\{ {F}(x) =  \sum_{i=1}^N f_i(x)\Big\},
\end{equation}
with convex (possibly nondifferentiable) functions  $f_i:\Rb^n\to\Rb$, $i=1,\dots,N$,
where the number of component functions, $N$, may be arbitrarily large.
We only assume that the minimum exists.

In a typical application, $f_1(\cdot)$ may be the loss function, similar to problem \eqref{problem1},
while the penalty function is a sum of multiple components. This type of generalization has many practical applications, including low rank matrix completion, compressed sensing, dynamic network analysis, and computer vision.

To the best of the authors' knowledge,
no general convergent versions of operator splitting methods for multiple blocks exist. A
known way is to introduce $N$ copies $x^1=x^2=\dots=x^N$ of $x$, and reduce the problem to the two-function case in the space $\Rb^{nN}$ \cite{Combettes-Pesquet}:
\[
\min \sum_{i=1}^N f_i(x^i) + I(x^1,\dots,x^N)
\]
with $I(\cdot)$ denoting the indicator function of the subspace $x^1=x^2=\dots=x^N$. Similar ideas were used in stochastic programming, under the
name of \emph{Progressive Hedging} \cite{rockafellar1991scenarios}. A method for three blocks with one
function being differentiable was theoretically analyzed in \cite{condat2013primal}.

Our new algorithm, which we call the \emph{Selective Linearization Method} (SLIN), does not
replicate the decision variables. It generates a  sequence of points $x^k\in \Rb^n$ with
a monotonic sequence of corresponding function values $\big\{F(x^k)\big\}$.
At each iteration, it linearizes all but one
of the component functions and uses a proximal term penalizing for the distance to the last iterate.
In a sense, each step is a backward step of the form employed in operator splitting. The order of
processing the functions is not fixed; the method uses precise criteria for selecting the function to be treated exactly at the current step. It also employs
special rules for updating the proximal center. These two rules differ our approach from the simultaneously proposed
incremental proximal method of \cite{bertsekas-report}, which applies to smooth functions only, and achieves linear
convergence rate in this case.

The algorithm is a multi-block extension of the Alternating Linearization method for solving two-block nonsmooth optimization problems. Global convergence and convergence rate of the new algorithm are proved. Specifically, the new algorithm is proven to require at most $\mathcal{O}\big(\ln(1/\varepsilon)/\varepsilon\big)$ iterations to achieve solution accuracy
$\varepsilon$.

In section \ref{s:method}, we present the method and prove its global convergence. The convergence rate is derived in section \ref{s:rate}. Finally, in section \ref{s:application}, we
illustrate its operation on structured regularized regression problems involving many blocks.

\section{The Method}
\label{s:method}

Our method derives from two fundamental ideas of convex optimization:  the \emph{Moreau--Yosida regularization} of $F(\cdot)$,
\begin{equation}
\label{MY-regularization}
F_D(y) = \min \Big\{ F(x) +  \frac{1}{2}\big\|x-y\big\|^2_D \Big\},
\end{equation}
and the \emph{proximal step} for \eqref{problem},
\begin{equation}
\label{prox}
\prox_F(y) = \argmin \Big\{ F(x) +  \frac{1}{2}\big\|x-y\big\|^2_D \Big\}.
\end{equation}
In the formulas above, the norm $\|x\|_{_{D}} = \big(\langle x,D x\rangle\big)^{1/2}$ with a positive definite matrix $D$.
In applications, we shall use a diagonal $D$, which leads to major computational simplifications. The \emph{proximal point method} carries out the
iteration $x^{k+1} = \prox_F(x^k)$, $k=1,2,\dots$ and is known to converge to a minimum of $F(\cdot)$, if a minimum exists \cite{rockafellar1976monotone}.

The main idea of our method is to replace problem \eqref{problem} with a sequence of approximate problems
of the following form:
\begin{equation}
\label{subproblem}
\min_x f_{j_k}(x) + \sum_{i\neq j_k}\tilde{f}^{k}_i(x) + \frac{1}{2}\big\|x-x^k\big\|^2_D.
\end{equation}
Here $k=1,2,\dots$ is the iteration number, $x^k$ is the current best approximation to the solution, $j_k\in \{1,\dots,N\}$
is an index selected at iteration $k$, and
$\tilde{f}_i^{k}$ are \emph{affine minorants} of the functions $f_i$, $i\in \{1,\dots,N\}\setminus \{j_k\}$. These minorants are constructed as follows:
\[
\tilde{f}^{k}_i(x) = f_i(z^{k}_{i}) + \langle g^{k}_{i}, x - z^k_{i}\rangle,
\]
with some points $z^{k}_i\in \Rb^n$ and specially selected subgradients
$g^{k}_{i} \in \partial f_i(z^{k}_i)$. Thus, problem \eqref{subproblem} differs from the proximal point problem in
\eqref{MY-regularization} by the fact that only one of the functions $f_i(\cdot)$ is treated exactly, while the other functions are replaced by affine approximations.

The key elements of the method are the selection of the index $j_k$, the way the affine approximations are constructed,
and the update rule for the proximal center $x^k$. In formula \eqref{subproblem} and in the algorithm description below
we write simply $i\ne j_k$ for $i\in \{1,\dots,N\}\setminus\{j_k\}$.

We denote the function approximating $F(x)$ in \eqref{subproblem} by
\[
\tilde{F}^k(x) = f_{j_k}(x) + \sum_{i\neq j_k}\tilde{f}^{k}_i(x).
\]

\noindent
\textbf{Selective Linearization (SLIN) Algorithm}\vspace{1ex}\\
\textbf{Step 0:} Set $k=1$ and $j_1\in \{1,\dots,N\}$, select $x^1 \in \Rb^n$ and, for all $i\ne j_1$, linearization points
$z^1_{i}\in \Rb^n$ where the corresponding subgradients $g^1_{i}\in \partial f_i(z^1_{i})$ exist. Define $\tilde{f}^1_i(x) = f_i(z^1_{i}) + \langle g^1_{i}, x - z^1_{i}\rangle$ for $i\ne j_1$. Choose parameters $\beta \in (0,1)$, and a stopping precision $\varepsilon > 0$. \vspace{1ex}\\
\textbf{Step 1:} Find the solution $z^k_{j_k}$ of the $f_{j_k}$-subproblem \eqref{subproblem} and
define
\begin{equation}
\label{subgradient-selection}
g^k_{j_k} = -\sum_{i\ne j_k}g^{k}_i - D(z^k_{j_k}-x^k).
\end{equation}
\textbf{Step 2:}  If
\begin{equation}\label{eqn: n_v}
F(x^k) - \tilde{F}^k(z^k_{j_k}) \le \varepsilon,
\end{equation}
then stop. Otherwise, continue.\vspace{1ex}\\
\textbf{Step 3:} If
\begin{equation}\label{eqn: n_new_rule1}
F(z^k_{j_k}) \leq F(x^k) - \beta\big(F(x^k) - \tilde{F}^k(z^k_{j_k})\big),
\end{equation}
then set $x^{k+1} = z^k_{j_k}$ (\emph{descent step}); otherwise set $x^{k+1}=x^k$ (\emph{null step}).\vspace{1ex}\\
\textbf{Step 4:} Select
\begin{equation}
\label{j-selection}
j_{k+1} = \argmax_{i\neq j_k}\big\{f_i(z^k_{j_k}) - \tilde{f}^{k}_i(z^k_{j_k})\big\}.
\end{equation}
For all $i\ne j_{k+1}$, set $z^{k+1}_{i} = z^{k}_{i}$ and $g^{k+1}_{i} = g^{k}_{i}$ (so that $\tilde{f}^{k+1}_i(\cdot) \equiv \tilde{f}^{k}_i(\cdot)$). Increase $k$ by~1 and go to Step 1.\vspace{1ex}

Few comments are in order. Since the point $z^k_{j_k}$ is a solution of the subproblem \eqref{subproblem}, the
vector $g^k_{j_k}$ calculated in \eqref{subgradient-selection} is indeed a subgradient of $f_{j_k}$ at $z^k_{j_k}$; in fact, it is exactly the subgradient that features in the optimality condition for \eqref{subproblem} at $z^k_{j_k}$. Therefore,
at all iterations, the functions $\tilde{f}^{k}_i(\cdot)$ are minorants of the functions $f_i(\cdot)$. This in turn implies that $\tilde{F}^k(\cdot)$ is a lower approximation of $F(\cdot)$. Consequently, $F(x^k) - \tilde{F}^k(z^k_{j_k}) \ge 0$ in \eqref{eqn: n_v}, with $F(x^k) = \tilde{F}^k(z^k_{j_k})$ equivalent to $x^k$ being the minimizer of $F(\cdot)$.

In practical implementation of the algorithm, the points $z_i^k$ need not be stored. It is sufficient to
memorize $\alpha_i^k = f_i(z_i^k) - \langle g_i^k,z_i^k \rangle$ and the subgradients $g_i^k$. At Step 4, we then set $\alpha_i^{k+1} = \alpha_i^k$ and
$g_i^{k+1}=g_i^k$ for all $i\in \{1,\dots,N\}\setminus\{j_k, j_{k+1}\}$, while
$\alpha_{j_{k}}^{k+1} = f_i(z_i^k) - \langle g_i^k,z_i^k \rangle$. For $j_{k+1}$ these data are not needed, because the
function $f_{j_{k+1}}(\cdot)$ will not be linearized at the next iteration.

In some cases, the storage of the subgradients $g_i^k$ may be substantially simplified.
\begin{example}
\label{e:sumphi}
{\rm
Suppose
\[
F(x) = \sum_{i=1}^N \varphi_i(a_i^T x),
\]
with convex functions $\varphi_i:\Rb\to\Rb$ and $a_i\in \Rb^n$, $i-1,\dots,n$. Then every subgradient of
$f_i(x) = \varphi_i(a_i^T x)$ has the form $g_i^k = \sigma_i^k a_i$, with $\sigma_i^k\in \partial \varphi_i(a_i^T z_i^k)$.
The scalars $\sigma_i^k$ are sufficient for recovering the subgradients, because the vectors $a_i$ are part of the
problem data.
}
\end{example}

\section{Global convergence}
 We assume that $\varepsilon=0$ in  Step 2. To prove convergence of the algorithm,
 we consider two cases:  with finitely or infinitely many descent steps.

 We first address the finite case and show that the proximal center updated in the last descent step must be an optimal solution to problem \eqref{problem}. To this end,
we prove that if a null step is made at iteration $k$, then the optimal objective function values of consecutive subproblems are increasing and the gap is bounded below by a value determined by
\begin{equation}
\label{vk-def}
v_k = F(x^k) - \tilde{F}^k(z^k_{j_k}).
\end{equation}
We shall also use this result in the proof of convergence rate.

We denote the optimal objective function value of subproblem \eqref{subproblem} at iteration $k$ by
\[
\eta^k=\min_x f_{j_k}(x) + \sum_{i\neq j_k}\tilde{f}^{k}_i(x) + \frac{1}{2}\big\|x-x^k\big\|^2_D.
\]
\begin{lemma}\label{lem: increment_bound}
If a null step is made at iteration $k$, then
\begin{equation}
\label{eta-increase}
\eta^{k+1} \ge \eta^k + \frac{1-\beta}{2(N-1)}\bar{\mu}_kv_k,
\end{equation}
where
\begin{equation}
\label{muk-def}
\bar{\mu}_k =
\min\Big\{1, \frac{(1-\beta)v_k}{(N-1)\|s_{j_{k+1}}^k - g^{k}_{j_{k+1}}\|_{D^{-1}}^2}\Big\},
\end{equation}
with an arbitrary $s_{j_{k+1}}^k\in\partial f_{j_{k+1}}(z^k_{j_k})$.
\end{lemma}
\begin{proof}
The change from the $f_{j_k}$-subproblem  to the $f_{j_{k+1}}$-subproblem can be viewed as two steps: first is the change of $f_{j_k}(\cdot)$ to $\tilde{f}^k_{j_k}(\cdot)$, followed by the change of $\tilde{f}^k_{j_{k+1}}(\cdot)$ to $f_{j_{k+1}}(\cdot)$. By the selection of the subgradient \eqref{subgradient-selection} and the
resulting
construction of $\tilde{f}^k_{j_k}(\cdot)$, the first operation does not change the solution and the optimal value of the subproblem. Thus the optimal value of \eqref{subproblem} satisfies the following equation:
\begin{equation}
\label{etak-all-lin}
\eta^k =
\min_x \sum_{i=1}^N\tilde{f}^{k}_i(x) + \frac{1}{2}\big\|x-x^k\big\|_D^2.
\end{equation}
Since $x^{k+1} = x^k$ at a null step, and $f_{j_{k+1}} \geq \tilde{f}^k_{j_{k+1}}$,
the second operation can only increase the optimal value of the last problem. Therefore, $\eta^{k+1} \ge \eta^k$.

Consider the family of relaxations of the $f_{j_{k+1}}$-subproblem at iteration $k+1$:
\begin{equation}\label{eqn: f2_relax}
\begin{split}
\hat{Q}_k(\mu)=\min_x \biggl\{ &\sum_{i\neq j_{k+1}}\tilde{f}^{k+1}_i(x) + (1-\mu)\left(f_{j_{k+1}}(z^{k}_{j_{k+1}}) + \langle g^{k}_{j_{k+1}}, x - z^{k}_{j_{k+1}}\rangle\right)\\
                &+ \mu\left(f_{j_{k+1}}(z^k_{j_k}) + \langle s_{j_{k+1}}^k, x-z^k_{j_k}\rangle\right) +  \frac{1}{2}\big\|x-x^k\big\|_D^2\biggr\},
\end{split}
\end{equation}
with parameter $\mu\in [0,1]$. In the above relaxation, the function $f_{j_{k+1}}(\cdot)$ is replaced by a convex combination of its two affine minorants: one at the point $z^{k}_{j_{k+1}}$, which is $\tilde{f}^k_{j_{k+1}}(\cdot)$ used at itera\-tion~$k$,
and the other one at the $k$th trial point $z^k_{j_k}$, with an arbitrary subgradient
$s_{j_{k+1}}^k$.
Due to \eqref{etak-all-lin},  the value of (\ref{eqn: f2_relax}) with $\mu = 0$ coincides with~$\eta^k$.
Therefore, the difference between $\eta^{k+1}$ and $\eta^k$ can be estimated from below by the increase in the optimal value $\hat{Q}_k(\mu)$ of (\ref{eqn: f2_relax}) when $\mu$ moves away from zero. That is,
\begin{equation}
\eta^{k+1} - \eta^k \geq \max_{\mu\in[0,1]}\hat{Q}_k(\mu) - \hat{Q}_k(0).
\end{equation}

Define $\delta_k = F(z^k_{j_k}) - \tilde{F}^k(z^k_{j_k})$. Note that $\delta_k \geq 0$, since $f_i \geq \tilde{f}^{k}_i$ for $i\neq j_k$. We also define ${\mu}_k = \min\Big\{1, \frac{\delta_k}{(N-1)\big\|s_{j_{k+1}}^k - g^{k}_{j_{k+1}}\big\|^2_{D^{-1}}}\Big\}$, so $ {\mu}_k\in[0,1]$. By direct calculation,
and with the use of \eqref{subgradient-selection}, the solution of (\ref{eqn: f2_relax}) has the form
\[
\hat{x}(\mu)=x^k-D^{-1}\Big[\sum_{i=1}^N g_i^{k} + \mu\big(s_{j_{k+1}}^k-g^{k}_{j_{k+1}}\big)\Big]
= z^k_{j_k}+ \mu D^{-1}\big(s_{j_{k+1}}^k-g^{k}_{j_{k+1}}\big).
 \]
Using the definitions following (\ref{eqn: f2_relax}) and the fact that $\hat{x}(0) = z^k_{j_k} $, the derivative of $\hat{Q}_k$ can be expressed as follows:
\begin{equation}
\begin{aligned}
\hat{Q}'_k(\mu) &= \langle s_{j_{k+1}}^k - g^{k}_{j_{k+1}}, \hat{x}(\mu) \rangle\\
&{\qquad} + \left(f_{j_{k+1}}(z^k_{j_k}) - \langle s_{j_{k+1}}^k, z^k_{j_k}\rangle\right) - \left(f_{j_{k+1}}(z^{k}_{j_{k+1}}) - \langle g^{k}_{j_{k+1}}, z^{k}_{j_{k+1}}\rangle\right) \\
&= \langle s_{j_{k+1}}^k - g^{k}_{j_{k+1}}, \hat{x}(\mu) - z^k_{j_k} \rangle \\
&{\qquad} + f_{j_{k+1}}(z^k_{j_k}) - \left(f_{j_{k+1}}(z^{k}_{j_{k+1}}) + \langle g^{k}_{j_{k+1}}, z^k_{j_k} - z^{k}_{j_{k+1}}\rangle\right) \\
&\geq \langle s_{j_{k+1}}^k - g^{k}_{j_{k+1}}, \hat{x}(\mu) - z^k_{j_k} \rangle + \frac{F(z^k_{j_k}) - \tilde{F}^k(z^k_{j_k})}{N-1}\\
&= -\mu \big\|s_{j_{k+1}}^k - g^{k}_{j_{k+1}}\big\|^2_{D^{-1}} + \frac{\delta_k}{N-1}.
\end{aligned}
\end{equation}
In the inequality above, we used the definition \eqref{j-selection}  of $j_{k+1}$ and the fact that the maximum of the
differences $f_{j}(z^k_{j_k})-\tilde{f}^k_{j}(z^k_{j_k})$ over $j\ne j_{k}$ is larger than their average.
Thus
\begin{equation}
\hat{Q}_k({\mu}_k)-\hat{Q}_k(0)= \int^{{\mu}_k}_0 \hat{Q}'_k(\mu)\,d\mu \ge  {\mu}_k\left(\frac{\delta_k}{N-1} - \frac{1}{2}{\mu}_k\big\|s_{j_{k+1}}^k - g^{k}_{j_{k+1}}\big\|^2_{D^{-1}}\right).
\end{equation}
Substitution of the definition of ${\mu}_k$ yields
\begin{equation}\label{eqn: gap}
\eta^{k+1} \geq \eta^k + \frac{{\mu}_k\delta_k}{2(N-1)}.
\end{equation}
If a null step is made at iteration $k$, then the update step rule (\ref{eqn: n_new_rule1}) is violated.
Thus, $\delta_k = F(z^k_{j_k}) - \tilde{F}^k(z^k_{j_k}) > (1 - \beta)v_k.$
Plugging this lower bound on $\delta_k$ into (\ref{eqn: gap}) and using the definition of $\bar{\mu}_k$, we obtain the
postulated bound \eqref{eta-increase}.

Finally, we  remark that $s_{j_{k+1}}^k \ne g^{k}_{j_{k+1}}$, because $f_{j_{k+1}}(z^k_{j_k})>\tilde{f}^k_{j_{k+1}}(z^k_{j_k})$.
\end{proof}

We also need to estimate the size of the steps made by the method.
\begin{lemma}
\label{l:radius}
At every iteration $k$,
\begin{equation}
\label{radius2}
\frac{1}{2}\big\|z^k_{j_k} - \prox_F(x^{k})\big\|^2_D \le F_D(x^{k}) - \eta^k.
\end{equation}
\end{lemma}
\begin{proof}
Since $F(\cdot) \ge \tilde{F}^k(\cdot)$ and $z^k_{j_k}$ is a solution of the strongly convex problem \eqref{subproblem}, we have
\begin{equation}
\label{estm-dist}
\begin{aligned}
F_D(x^{k}) &= F\big(\prox_F(x^{k})\big) +  \frac{1}{2}\big\|\prox_F(x^{k})-x^{k}\big\|^2_D\\
&\ge
\tilde{F}^k\big(\prox_F(x^{k})\big)  +  \frac{1}{2}\big\|\prox_F(x^{k})-x^{k}\big\|^2_D\\
&\ge
\tilde{F}^k(z^k_{j_k}) +  \frac{1}{2}\big\|z^k_{j_k}-x^{k}\big\|^2_D + \frac{1}{2}\big\|z^k_{j_k} - \prox_F(x^{k})\big\|^2_D\\
&= \eta^k + \frac{1}{2}\big\|z^k_{j_k} - \prox_F(x^{k})\big\|^2_D.
\end{aligned}
\end{equation}
Rearranging, we obtain \eqref{radius2}.
\end{proof}

We are now ready to prove optimality in the case of finitely many descent steps.
\begin{theorem}
\label{t:finite-descent}
Suppose $\varepsilon=0$, the set $\mcK = \{1\} \cup \{k>1: x^{k} \neq x^{k-1}\}$ is finite  and $\inf F > -\infty$. Let $m \in \mcK$ be the largest index such that $x^{m}\neq x^{m-1}$. Then $x^{m} \in \text{\rm Argmin}\,F$.
\end{theorem}
\begin{proof}
We argue by contradiction. Suppose $x^{m} \notin {\rm Argmin}\, F$. If $\varepsilon=0$ the method cannot stop, because
$\tilde{F}^k(z^k_{j_k}) \le F(\prox_F(x^m)) < F(x^m)$, for all $k \ge m$.
Therefore, null steps are made at all iterations $k \ge m$, with $x^k = x^{m}$.
By Lemma \ref{lem: increment_bound}, the sequence $\{\eta^k\}$ is nondecreasing and bounded above by $F(x^{m})$. Hence $\eta^{k+1} - \eta^k \rightarrow 0$. The right hand side of estimate \eqref{radius2} with $x^k=x^m$ for $k\ge m$,
owing to the monotonicity of $\{\eta^k\}$, is nonincreasing, and thus
the sequence $\{z^k_{j_k}\}$ is bounded. Since the subgradients of a convex function are locally bounded,
 the differences $\big\|s_{j_{k+1}}^k - g^{k}_{j_{k+1}}\big\|_{D^{-1}}$ appearing in the definition of $\bar{\mu}_k$ in Lemma \ref{lem: increment_bound} are bounded from above. Therefore,
$v_k \rightarrow 0$. As $F(x^{m}) \geq \eta^k \geq F(x^{m}) -v_k$, we have $\eta^k \uparrow F(x^{m})$.

On the other hand, the inequality $\tilde{F}^k(\cdot) \le F(\cdot)$ implies that
$\eta^k \leq F_{D}(x^{m})$ for all $k \geq m$. Since $x^{m} \notin {\rm Argmin}F$, we have $F_{D}(x^{m}) < F(x^{m})$,
which contradicts the convergence of $\{\eta^k\}$ to $F(x^{m})$.
\end{proof}

We now address the infinite case. Note that the update test (\ref{eqn: n_new_rule1}) can be expressed as follows:
\begin{equation}
\label{eqn: n_new_rule2}
\tilde{F}^k(z^k_{j_k}) \geq -\frac{1}{\beta}F(z^k_{j_k}) + \frac{1-\beta}{\beta}F(x^k).
\end{equation}
\begin{theorem}
\label{t:infinite-descent}
Suppose ${\rm Argmin}\,F\ne \emptyset$. If the set $\mcK = \{k: x^{k+1} \neq x^k\}$ is infinite,
then $\lim_{k \rightarrow \infty}x^k = x^*$, for some $x^* \in {\rm Argmin}\,F$.
\end{theorem}
\begin{proof}
Consider iteration $k \in \mcK$ (descent step). From the optimality condition for \eqref{subproblem}  we obtain
\begin{equation}
0 \in \partial \Big[ f_{j_k}(z^k_{j_k}) + \sum_{i\neq j_k}\tilde{f}^{k}_i(z^k_{j_k})\Big] + D\big(z^k_{j_k} - x^k\big),
\end{equation}
which yields
\begin{equation}\label{6-bundle-optimality}
 D(x^k - x^{k+1})\in\partial\Big[ f_{j_k}(z^k_{j_k})+ \sum_{i\neq j_k}\tilde{f}^{k}_i(z^k_{j_k})\Big].
\end{equation}
Then for any point $x^*\in {\rm Argmin}\,F $ we obtain
\begin{align}
\label{6-bundle-subdif}
F(x^*) \geq \tilde{F}^k(x^*) \geq \tilde{F}^k(x^{k+1}) + \big\langle D(x^k-x^{k+1}), x^*-x^{k+1}\big\rangle.
\end{align}
Hence
\begin{align*}
\big\|x^{k+1}-x^*\big\|_D^2 & = \big\|x^k-x^*\big\|_D^2 + 2\big\langle D(x^{k+1}-x^k), x^k-x^*\big\rangle +
\big\|x^{k+1}-x^k\big\|_D^2\\
&\leq \big\|x^k-x^*\big\|_D^2 + 2\big\langle D(x^{k+1}-x^k), x^{k+1}-x^*\big\rangle\\
&\leq \big\|x^k - x^*\big\|_D^2 + 2\big(F(x^*) - \tilde{F}^k(x^{k+1})\big).
\end{align*}
Using \eqref{eqn: n_new_rule2}, we can continue this chain of inequalities as follows
\begin{equation}\label{6-bundle-scalarproduct}
\begin{aligned}
\big\|x^{k+1}-x^*\big\|_D^2 & \le
\big\|x^k - x^*\big\|_D^2 + 2\big(F(x^*) - \frac{1}{\beta}F(x^{k+1}) + \frac{1-\beta}{\beta}F(x^k)\big) \\
&=\big\|x^k-x^*\big\|_D^2 + 2\big(F(x^*) - F(x^k)\big) + \frac{2}{\beta}\big(F(x^k) -F(x^{k+1})\big).
\end{aligned}
\end{equation}
Thus, adding up (\ref{6-bundle-scalarproduct}) for all $k\in \mcK$, $k\leq m$, and noting that the null steps do not change the proximal centers, we obtain
\begin{equation}\label{6-bundle-key}
\big\|x^{m+1} - x^*\big\|_D^2 \leq \big\|x^1 - x^*\big\|_D^2 + 2\sum_{{k\in \mcK}\atop{ k\leq m}}\big(F(x^*) - F(x^{k})\big) + \frac{2}{\beta}\sum_{{k\in\mcK}\atop{ k\leq m}}\big(F(x^{k}) - F(x^{k+1})\big).
\end{equation}
The term $2\sum_{k\in\mcK, k\leq m}\big(F(x^*) - F(x^{k})\big)$ is non-positive, and the last term is bounded by $\frac{2}{\beta}\big(F(x^1)-F(x^*)\big)$. 
Thus, several conclusions follow from inequality \eqref{6-bundle-key}. First, the sequence $\{x^k\}_{k\in\mcK}$ is bounded, because their distances to $x^*$ are bounded. Secondly, rewriting (\ref{6-bundle-key}) as
\[
\sum_{k\in\mcK,k\leq m}\big(F(x^{k}) - F(x^*)\big) \leq \frac{1}{2}\big(\big\|x^1 - x^*\big\|_D^2 - \big\|x^{m+1} - x^*\big\|^2_D\big) + \frac{1}{\beta}\big(F(x^1) - F(x^{m+1})\big),
\]
and letting $m\rightarrow \infty$ in $\mcK$, we deduce that
\begin{equation}\label{6-penalty-sumup}
\sum_{k\in\mcK}\left(F(x^{k})-F(x^*)\right)\leq \frac{1}{2}\big\|x^1-x^*\big\|^2_D + \frac{1}{\beta}\big(F(x^1)-F(x^*)\big).
\end{equation}
Consequently, $F(x^k) \rightarrow F(x^*)$ as $k\rightarrow \infty$ in $\mcK$. As the null steps do not change the proximal centers, we also have $F(x^k)\rightarrow F(x^*)$, when $k\rightarrow \infty$.

To prove that the sequence of proximal centers converges to an optimal solution, note that since the infinite sequence $\{x^k\}_{k\in\mcK}$ is bounded, it has a convergent subsequence whose limit $\hat{x}$ is a minimizer of $F$.
 Without loss of generality, we  substitute $\hat{x}$ for $x^*$ in the above derivations,
and add (\ref{6-bundle-scalarproduct}) for all $k\in \mcK$ such that $\ell \le k\leq m$.
For any $1 \le \ell \le m$ we obtain the following analog of \eqref{6-bundle-key}:
\begin{align*}
\big\|x^{m+1} - \hat{x}\big\|_D^2 &\leq \big\|x^{\ell} - \hat{x}\big\|_D^2 + 2\sum_{{k\in \mcK}\atop{ k\leq m}}\big(F(\hat{x}) - F(x^{k})\big) + \frac{2}{\beta}\sum_{{k\in\mcK}\atop{ k\leq m}}\big(F(x^{k}) - F(x^{k+1})\big) \\
&\leq \big\|x^{\ell} - \hat{x}\big\|_D^2 + \frac{2}{\beta}\big(F(x^{\ell}) - F(\hat{x})\big).
\end{align*}
The right hand side of the last inequality can be made arbitrarily small by choosing $\ell$ from the subsequence converging to $\hat{x}$. Therefore the entire sequence $\{x^k\}_{k\in\mcK}$ is convergent to $\hat{x}$.
\end{proof}

We finish this section with a number of conclusions, which will be useful in the analysis of the rate of convergence.

\begin{lemma}
\label{l:eta-in-descent} If there is a descent step at iteration $k$, then
\begin{equation}
\label{eta-descent}
\eta^{k+1} - \eta^{k} \ge - \big\|x^{k+1}-x^k\big\|^2_D  \ge   \frac{1}{\beta} \big( F(x^{k+1}) - F(x^k)\big).
\end{equation}
\end{lemma}
\begin{proof}
By \eqref{subgradient-selection},
\begin{equation}
\label{key-subgradient}
\sum_{i=1}^N g^{k}_i + D(x^{k+1} - x^k) = 0.
\end{equation}
The optimal value of \eqref{subproblem} at iteration $k+1$ can be then estimated as follows:
\begin{align*}
\eta^{k+1}  &= \min_x\bigg\{ f_{j_{k+1}}(x) + \sum_{i\ne j_{k+1}} \tilde{f}_i^{k+1}(x) + \frac{1}{2}\big\|x-x^{k+1}\big\|^2_D\bigg\} \\
&\ge \min_x\bigg\{ \sum_{i=1}^N \tilde{f}_i^{k+1}(x) + \frac{1}{2}\big\|x-x^{k+1}\big\|^2_D\bigg\}\\
&= \sum_{i=1}^N \tilde{f}_i^{k+1}(x^{k+1})+ \min_x\bigg\{\Big\langle \sum_{i=1}^N g^{k}_i, x - x^{k+1}\Big\rangle
+ \frac{1}{2}\big\|x-x^{k+1}\big\|^2_D\bigg\}\\
&= \tilde{F}^{k}(x^{k+1})+ \min_x\bigg\{-\big\langle D(x^{k+1} - x^k) , x - x^{k+1}\big\rangle
+ \frac{1}{2}\big\|x-x^{k+1}\big\|^2_D\bigg\}.
\end{align*}
The minimizer on the right hand side is $x = 2x^{k+1}-x^k$, and we conclude that
\[
\eta^{k+1} \ge \tilde{F}^{k}(x^{k+1}) -  \frac{1}{2}\big\|x^{k+1}-x^k\big\|^2_D = \eta^k - \big\|x^{k+1}-x^k\big\|^2_D,
\]
which proves the left inequality in \eqref{eta-descent}. To prove the right inequality,
we observe that the test \eqref{eqn: n_new_rule1} for the
descent step is satisfied at iteration $k$, and thus
\[
F(x^k) - F(x^{k+1}) \ge  \beta\big(F(x^k) - \tilde{F}^k(x^{k+1})\big)
= \beta\big(\tilde{F}^k(x^{k}) - \tilde{F}^k(x^{k+1})\big).
\]
The expression on the right hand side can be calculated with the use of \eqref{key-subgradient}, exactly as in the
derivations above, which yields
\[
F(x^k) - F(x^{k+1}) \ge \beta \big\|x^{k+1}-x^k\big\|^2_D.
\]
This proves the right inequality in \eqref{eta-descent}.
\end{proof}

We can now summarize convergence properties of the sequences generated by the algorithm.
\begin{corollary}
\label{c:summary}
Suppose ${\rm Argmin}\,F\ne \emptyset$ and $\varepsilon=0$. Then a point $x^*\in {\rm Argmin}\,F$ exists, such that:
\begin{tightlist}{ii}
\item $\displaystyle{\lim_{k\to\infty} x^k = \lim_{k\to\infty} z^k_{j_k} = x^*}$;
\item $\displaystyle{\lim_{k\to\infty} \eta^k = F(x^*)}$.
\end{tightlist}
\end{corollary}
\begin{proof} The convergence of $\{x^k\}$ to a minimum point $x^*$ has been proved in
Theorems \ref{t:finite-descent} and \ref{t:infinite-descent}. It remains to verify the convergence properties of $\{z^k_{j_k}\}$
and $\{\eta^k\}$. It follows from Lemmas \ref{lem: increment_bound} and  \ref{l:eta-in-descent} that the sequence
$\eta^k-\frac{1}{\beta}F(x^k)$ is nondecreasing.
Since $\eta^k \le F(x^k)$ by construction, this sequence is bounded from above, and thus convergent. Therefore, a limit $\eta^*$
of $\{\eta^k\}$ exists and $\eta^* \le F(x^*)$.
If the number of descent steps is finite, the equality $\eta^* = F(x^*)$ follows from
Theorem \ref{t:finite-descent}. If the number of descent steps is infinite, inequality
\eqref{eqn: n_new_rule1} at each descent step $k$ yields:
\[
F(x^k) - \eta^k \le F(x^k) - \tilde{F}^k(x^{k+1}) \le \frac{1}{\beta} \big( F(x^k) - F(x^{k+1}) \big).
\]
Passing to the limit over descent steps $k\to \infty$ we conclude that $\eta^* \ge F(x^*)$. Consequently, $\eta^*=F(x^*)$ and assertion (ii)
is true.

The convergence of the sequence $\{z^k_{j_k}\}$ to $x^*$ follows from inequality \eqref{radius2},
because $x^k \to x^*$ and $\eta^k \to F(x^*)$.
\end{proof}

\section{Rate of Convergence}
\label{s:rate}

Our objective in this section is to estimate the rate of convergence of the method. To this end, we assume that $\varepsilon>0$
at Step 2 (inequality \eqref{eqn: n_v}) and we estimate the number of iterations needed to achieve this accuracy.
We also make an additional assumption about the growth rate of the function $F(\cdot)$.
\begin{assumption}
\label{a:growth}
The function $F(\cdot)$ has a unique minimum point $x^*$ and a constant $\alpha >0$ exists, such that
\[
F(x) - F(x^{*}) \geq \alpha \big\|x - x^{*}\big\|^2_{D},
\]
for all $x\in \Rb^n$.
\end{assumption}

Assumption \ref{a:growth} has a number of implications on the properties of the method. First, we recall from
\cite[Lem. 7.12]{ruszczynski2006nonlinear} the following estimate of the Moreau--Yosida regularization.
\begin{lemma}
\label{l:lemma7.12} For any point $x\in \Rb^n$, we have
\begin{equation}
\label{lemma7.12}
F_D(x) \leq  F(x) - \big\|x - x^{*}\big\|^2_D\,\varphi\bigg( \frac{F(x) - F(x^{*})}{\big\|x - x^{*}\big\|^2_D}\bigg),
\end{equation}
where
\[
\varphi(t) =
\begin{cases} t^2 & \text{if\; $t\in[0,1]$},\\
-1 +2t & \text{if\; $t \ge 1$}.
\end{cases}
\]
\end{lemma}
\begin{proof} See \cite[Lem. 7.12]{ruszczynski2006nonlinear}. \end{proof}

\begin{lemma}\label{lem: conv_rate_1}
Suppose Assumption \ref{a:growth} is satisfied. Then the stopping test \eqref{eqn: n_v} implies that
\begin{equation}\label{eqn:stopping lemma}
F(x^k) - F(x^{*}) \leq \frac{\varepsilon}{\min(\alpha,1)}.
\end{equation}
\end{lemma}
\begin{proof}
As $\tilde{F}^k(\cdot) \leq F(\cdot)$, the stopping criterion implies that
\begin{equation}\label{eqn:stoping inequality}
\begin{aligned}
F_D(x^k) = \min_x\Big\{F(x) + \frac{1}{2}\big\|x - x^k\big\|^2_D \Big\} &\geq
\min_x \Big\{\tilde{F}^k(x) + \frac{1}{2}\big\|x - x^k\big\|^2_D \Big\}\\
 &= \tilde{F}^k (z^k_{j_k}) + \frac{1}{2}\big\|z^k_{j_k} - x^k\big\|^2_D \geq F(x^k) - \varepsilon.
\end{aligned}
\end{equation}
Consider two cases.\\
\emph{Case 1:} If $F(x^k) - F(x^{*}) \le \big\|x^k - x^{*}\big\|^2_D$, then \eqref{lemma7.12} with $x=x^k$ yields
\[
F_D(x^k) \leq  F(x^k) - \frac{\big(F(x^k) - F(x^{*})\big)^2}{\big\|x^k - x^{*}\big\|^2_D}.
\]
Combining this inequality with \eqref{eqn:stoping inequality}, we conclude that
\begin{equation}
 \frac{\big(F(x^k) - F(x^{*})\big)^2}{\big\|x^k - x^{*}\big\|^2_D} \leq \varepsilon.
\end{equation}
Substitution of the denominator by the upper estimate $(F(x^k) - F(x^{*}))/\alpha$ implies \eqref{eqn:stopping lemma}.\\
\emph{Case 2:}
$F(x^k) - F(x^{*})> \big\|x^k - x^{*}\big\|^2_D$.
Then \eqref{lemma7.12} yields
\[
F_D(x^k) \leq  F(x^k) - 2 \big( F(x^k) - F(x^{*}) \big) + \big\|x^k - x^{*}\big\|^2_D.
\]
With a view to \eqref{eqn:stoping inequality}, we obtain
\[
2 \big( F(x^k) - F(x^{*}) \big) - \big\|x^k - x^{*}\big\|^2_D \le \varepsilon,
\]
which implies that
$F(x^k) - F(x^{*}) \le \varepsilon$
in this case.
\end{proof}
\begin{lemma}
\label{l:FD-improvement}
Suppose Assumption \ref{a:growth} is satisfied. Then at any iteration $k$ we have
\[
F(x^k) - \eta^k \ge \frac{2\varphi(\alpha)}{1 + 2\varphi(\alpha)}\big( F(x^k) - F(x^*)\big).
\]
\end{lemma}
\begin{proof}
By Lemma \ref{l:radius},
\[
F(x^k) - \eta^k \ge F(x^k)-F_D(x^k).
\]
To derive a lower bound for the right hand side of the last inequality, we
use Assumption \ref{a:growth} in \eqref{lemma7.12} with $x=x^k$. We obtain
\begin{equation}
\label{lemma7.12a}
F_D(x^k) \leq  F(x^k) - \big\|x^k - x^{*}\big\|^2_D\,\varphi(\alpha).
\end{equation}
By the definition of the Moreau--Yosida regularization, for any optimal solution $x^*$ we have
\[
F(x^*) + \frac{1}{2} \big\| x^*-x^k\big\|^2_D \ge F_D(x^k),
\]
and thus
\[
\big\|x^k - x^{*}\big\|^2_D \ge 2\big(F_D(x^k)-F(x^*)\big).
\]
Substitution to \eqref{lemma7.12a} yields
\[
F(x^k) - F_D(x^k) \ge 2\big(F_D(x^k)-F(x^*)\big)\varphi(\alpha), 
\]
which can be manipulated to
\[
 F(x^k) - F_D(x^k) \ge \frac{2\varphi(\alpha)}{1 + 2\varphi(\alpha)}( F(x^k) - F(x^*)\big).
\]
This can be combined with the first inequality in the proof, to obtain the desired result.
\end{proof}

In order to estimate the number of iterations of the method needed to achieve the prescribed
accuracy, we need to consider two aspects. First, we prove linear rate of convergence between descent steps. Then, we estimate
 the numbers of null steps between consecutive descent steps.

By employing the estimate of Lemma \ref{lem: conv_rate_1}, we can address the first aspect. To simplify notation, with no loss of generality, we assume that $\alpha \in (0,1]$
(otherwise, we would have to replace $\alpha$ with $\bar{\alpha} = \min(\alpha,1)$ in the considerations below).

\begin{lemma}\label{lem: conv_rate_2}
Suppose $x^{*}$ is the unique minimum point of $F(\cdot)$ and  Assumption \ref{a:growth} is satisfied. Then at every descent step $k$, when the update step rule \eqref{eqn: n_new_rule1} is satisfied, we have the inequality:
\begin{equation}
\label{e:linear-rate}
F(z^k_{j_k}) - F(x^*) \leq (1 - \alpha\beta)\big(F(x^k) - F(x^{*})\big).
\end{equation}
\end{lemma}
\begin{proof}
It follows from the update rule (\ref{eqn: n_new_rule1}) that
\[
F(z^k_{j_k}) \leq F(x^k) - \beta\big(F(x^k) - \tilde{F}^k(z^k_{j_k})\big).
\]
Using Lemma \ref{lem: conv_rate_1} with $\varepsilon=F(x^k) - \tilde{F}^k(z^k_{j_k})$, we obtain
\[
F(x^k) - F(x^{*}) \leq \frac{1}{\alpha} \big(F(x^k) - \tilde{F}^k(z^k_{j_k})\big).
\]
Combining these inequalities and simplifying, we conclude that
\begin{align*}
F(z^k_{j_k}) &\leq (1 - \beta) F(x^k) + \beta \big( \alpha F(x^{*}) - \alpha F(x^k) + F(x^k)\big) \\
&=F(x^k) - \alpha\beta\big( F(x^k) - F(x^{*})\big).
\end{align*}
Subtracting $F(x^*)$ from both sides, we obtain the linear rate \eqref{e:linear-rate}.
\end{proof}

We now pass to the second issue: the estimation of the number of null steps between two consecutive
descent steps. We shall base it on the analysis of the gap $F(x^k)-\eta^k$.

By virtue of Corollary \ref{c:summary}, all points $\{z^k_{j_k}\}$ generated by the algorithm are uniformly bounded.
 Since subgradients of finite-valued convex functions are locally bounded, the subgradients of all $f_{j_k}$ are bounded,
and thus a constant $M$ exists, such that
\[
\big\| s^k_{j_{k+1}}-g^k_{j_{k+1}}\big\|^2_{D^{-1}} \le M
\]
at all null steps. With no loss of generality, we assume that $\varepsilon \le (N-1)M$.

\begin{lemma}\label{lem: conv_rate_3}
If a null step is made at iteration $k$, then
\begin{equation}
F(x^k) - \eta^{k+1} \le \gamma \big(F(x^k) - \eta^k\big),
\end{equation}
where
\begin{equation}
\label{tauk-bound}
\gamma = 1 - \frac{1}{2}\bigg(\frac{1-\beta}{N-1}\bigg)^2\frac{\varepsilon}{M}.
\end{equation}
\end{lemma}
\begin{proof}
By Lemma \ref{lem: increment_bound}, we have
\begin{equation}
F(x^k) - \eta^{k+1} \le F(x^k) - \eta^k - \frac{1-\beta}{2(N-1)}\bar{\mu}_kv_k.
\end{equation}
On the other hand,
\begin{equation}
v_k = F(x^k) - \tilde{F}^k(z^k_{j_k}) = F(x^k) - \eta^k + \frac{1}{2}\big\|z^k_{j_k} - x^k\big\|_D^2 \geq F(x^k) - \eta^k.
\end{equation}
Combining the last two inequalities, we conclude that
\begin{equation}
\label{gap-decrease}
\begin{split}
F(x^k) - \eta^{k+1} 
    &\le F(x^k) - \eta^k - \frac{1-\beta}{2(N-1)}\bar{\mu}_k\big(F(x^k) - \eta^k\big)\\
    &= \biggl(1-\frac{1-\beta}{2(N-1)}\bar{\mu}_k\biggr)\big(F(x^k) - \eta^k\big).
\end{split}
\end{equation}
Consider the definition \eqref{muk-def} of $\bar{\mu}_k$ in Lemma \ref{lem: increment_bound}.
If $\bar{\mu}_k = 1$, then $1- \frac{1-\beta}{2(N-1)}\bar{\mu}_k$ is no greater than the bound \eqref{tauk-bound},
because $\varepsilon\le (N-1)M$.
Otherwise, $\bar{\mu}_k$ is given by the second case in \eqref{muk-def}. Since the algorithm does not stop, we have $v_k > \varepsilon$, and thus
\[
 \bar{\mu}_k \ge \frac{(1-\beta)\varepsilon}{(N-1)M}.
\]
Substitution to \eqref{gap-decrease} yields \eqref{tauk-bound}.
\end{proof}

Let $x^{(\ell-1)}, x^{(\ell)}, x^{(\ell+1)}$ be three consecutive proximal centers in the algorithm ($\ell \le 2$).
We want to bound the number of iterations with the proximal center $x^{(\ell)}$. To this end, we bound two quantities:

1. The optimal objective value of the \emph{first} subproblem with proximal center $x^{(\ell)}$,
whose iteration number we denote by $k(\ell)$:
    \begin{equation}
    \label{etakl}
    \eta^{k(\ell)} = \min f_{j_{k(\ell)}}(x) + \sum_{i\neq j_{k(\ell)}}\tilde{f}^{k(\ell)}_i(x) + \frac{1}{2}\big\|x-x^{(\ell)}\big\|_D^2.
    \end{equation}
    We need an upper bound for $F(x^{(\ell)}) - \eta^{k(\ell)}$.

2.  The optimal objective value of the \emph{last} subproblem with proximal center $x^{(\ell)}$, occurring at iteration $k'(\ell)=
k(\ell+1)-1$:
    \begin{equation}
    \eta^{k'(\ell)} = \min f_{j_{k'(\ell)}}(x) + \sum_{i\neq j_{k'(\ell)}}\tilde{f}^{k'(\ell)}_i(x) + \frac{1}{2}\big\|x-x^{(\ell)}\big\|_D^2.
    \end{equation}
    We need an upper bound for $F(x^{(\ell)}) - \eta^{k'(\ell)}$ which implies the update rule (\ref{eqn: n_new_rule1}).

In the following we discuss each issue separately.

Recall that according to the algorithm, $x^{(\ell)}$ is the optimal solution of the last subproblem with proximal center $x^{(\ell-1)}$. Let $f_{j_{k(\ell)-1}}$ be the non-linearized component function of the last subproblem with proximal center $x^{(\ell-1)}$, whose optimal solution is $x^{(\ell)}$. The optimal value of the subproblem \eqref{subproblem} is
\begin{equation}
\label{etakl-1}
\eta^{k(\ell)-1} = f_{j_{k(\ell)-1}}(x^{\ell}) + \sum_{i\neq j_{k(\ell)-1}}\tilde{f}^{k(\ell)-1}_{i}(x^{(\ell)}) + \frac{1}{2}\big\|x^{(\ell)} - x^{(\ell-1)}\big\|_D^2.
\end{equation}

\begin{lemma}\label{lem: null_step_start_bound} If a descent step is made at iteration $k(\ell)-1$, then
\begin{equation}\label{eqn: null_step_upper_bound_1}
F(x^{(\ell)}) - \eta^{k(\ell)} \leq \frac{3}{2\beta}\big(F(x^{(\ell-1)}) - F(x^{(\ell)})\big).
\end{equation}
\end{lemma}
\begin{proof}
The left inequality in \eqref{eta-descent} yields
\[
\eta^{k(\ell)} \geq \eta^{k(\ell)-1} - \big\|x^{(\ell)} - x^{(\ell-1)}\big\|_D^2.
\]
Since $F(x^{(\ell)}) \leq F(x^{(\ell-1)})$, we obtain
\[
F(x^{(\ell)}) - \eta^{k(\ell)} \leq F(x^{(\ell-1)}) - \eta^{k(\ell)-1} + \big\|x^{(\ell)} - x^{(\ell-1)}\big\|_D^2.
\]
As iteration $k(\ell)-1$ is a descent step, the update rule (\ref{eqn: n_new_rule1}) holds. Thus
\begin{align*}
F(x^{(\ell-1)}) - \eta^{k(\ell)-1} &= \Big[F(x^{(\ell-1)}) - \tilde{F}^{k(\ell)-1}(x^{(\ell)})\Big]
- \frac{1}{2}\big\|x^{(\ell)} - x^{(\ell-1)}\big\|_D^2\\
 &\leq \frac{1}{\beta}\Big[F(x^{(\ell-1)}) - F(x^{(\ell)})\Big] - \frac{1}{2}\big\|x^{(\ell)} - x^{(\ell-1)}\big\|_D^2.
\end{align*}
Combining the last two inequalities we obtain
\[
F(x^{(\ell)}) - \eta^{k(\ell)} \leq \frac{1}{\beta}\big(F(x^{(\ell-1)}) - F(x^{(\ell)})\big)
+ \frac{1}{2}\big\|x^{(\ell)} - x^{(\ell-1)}\big\|_D^2.
\]
The right inequality in \eqref{eta-descent} can be now used to substitute $\big\|x^{(\ell)} - x^{(\ell-1)}\big\|_D^2$ on the right hand side
to obtain \eqref{eqn: null_step_upper_bound_1}.
\end{proof}

We can now integrate our results.

Applying Lemma \ref{l:FD-improvement}, we obtain the following inequality at \emph{every} null step with prox center~$x^{(\ell)}$:
\begin{equation}
\label{e:limitbd}
\begin{aligned}
F(x^{(\ell)}) - \eta^k &\ge \frac{2\varphi(\alpha)}{1 + 2\varphi(\alpha)}\big( F(x^{(\ell)}) - F(x^*)\big)\\
                     &\ge \frac{2\varphi(\alpha)}{1 + 2\varphi(\alpha)}\big( F(x^{(\ell)}) - F(x^{(\ell +1)})\big).
\end{aligned}
\end{equation}
From Lemma \ref{lem: null_step_start_bound} we know that for $\ell \ge 2$ the initial value of the left hand side (immediately after the
previous descent step) is bounded from above by the following expression:
\begin{equation}\label{e:ub1}
F(x^{(\ell)}) - \eta^{k(\ell)} \leq \frac{3}{2\beta}\big(F(x^{(\ell-1)}) - F(x^{(\ell)})\big).
\end{equation}
Lemma \ref{lem: conv_rate_3} established a linear rate of decrease of the left hand side. Therefore, the number $n_\ell$ of null steps
with proximal center $x^{(\ell)}$, if it is positive, satisfies the inequality:
\[
\frac{3}{2\beta}\big(F(x^{(\ell-1)}) - F(x^{(\ell)})\big) \gamma^{n_\ell - 1} \ge \frac{2\varphi(\alpha)}{1 + 2\varphi(\alpha)}\big( F(x^{(\ell)}) - F(x^{(\ell +1)})\big).
\]
Consequently, for $\ell \ge 2$ we obtain the following upper bound on the number of null steps:
\begin{equation}
\label{nl}
n_\ell \leq 1 + \frac{1}{\ln(\gamma)}  \ln\left(\frac{4\beta\varphi(\alpha)}{3(1 + 2\varphi(\alpha))}\frac{F(x^{(\ell)}) - F(x^{(\ell+1)})}{F(x^{(\ell-1)}) - F(x^{(\ell)})}\right).
\end{equation}
If the number $n_\ell$ of null steps is zero, inequality \eqref{e:linear-rate} yields
\begin{align*}
\frac{F(x^{(\ell)}) - F(x^{(\ell+1)})}{F(x^{(\ell-1)}) - F(x^{(\ell)})}
&\le \frac{F(x^{(\ell)}) - F(x^{*})}{F(x^{(\ell-1)})- F(x^*) - \big(F(x^{(\ell)})- F(x^*)\big)}
 \le
\frac{1}{\frac{1}{1-\alpha\beta}-1}.
\end{align*}
Elementary calculations prove that both logarithms on the right hand side of \eqref{nl} are negative, and thus inequality \eqref{nl} is satisfied in this case as well.

Suppose there are $L$ proximal centers appearing throughout the algorithm: $x^{(1)}$, $x^{(2)}$, \ldots, $x^{(L)}$. They divide the progress of the algorithm into $L$ series of null steps. For the first series, similarly to the analysis above, we use \eqref{e:limitbd} and Lemma \ref{lem: conv_rate_3} to obtain the estimate
\begin{equation}
\label{n1}
n_1 \leq 1 + \frac{1}{\ln(\gamma)}  \ln\left(\frac{2\varphi(\alpha)}{1 + 2\varphi(\alpha)}
\frac{F(x^{(1)}) - F(x^{(2)})}{F(x^{(1)}) - \eta^1}\right).
\end{equation}
For the last series, we observe that the inequality $F(x^{(\ell)}) - \eta^k \ge \varepsilon/2$ has to hold at each null step at which the stopping test was
not satisfied. We
 use it instead of \eqref{e:limitbd} and we obtain
\begin{equation}
n_L \leq 1 + \frac{1}{\ln(\gamma)}  \ln\left(\frac{\beta}{3}\frac{\varepsilon}{F(x^{(L-1)}) - F(x^{(L)})}\right).
\end{equation}
We aggregate the total number of null steps for different proximal centers throughout the algorithm
and we obtain the following bound:
\begin{equation}\label{null step bound}
\begin{split}
\sum_{\ell=1}^L n_\ell &= \frac{L-1}{\ln(\gamma)}\left[\ln\left(\frac{2\varphi(\alpha)}{1 + 2\varphi(\alpha)}\right)
+ \ln\left(\frac{\beta}{3}\right)
+ \frac{1}{L-1}\ln\left( \frac{\varepsilon}{F(x^{(1)}) - \eta^1}\right)\right]
 + L
\end{split}
\end{equation}
Let us recall the definition of $\gamma$ in \eqref{tauk-bound}, and denote
\[
C = \frac{1}{2}\bigg(\frac{1-\beta}{N-1}\bigg)^2\frac{1}{M},
\]
so that $\gamma = 1 - \varepsilon {C}$. Since
$\ln(1-\varepsilon C)< -\varepsilon C$, we derive the following inequality for the number of null steps:
\begin{equation}
\label{null step bound3}
\sum_{\ell=1}^L n_\ell \le \frac{L-1}{-\varepsilon C}\left[\ln\left(\frac{2\varphi(\alpha)}{1 + 2\varphi(\alpha)}\right)
+ \ln\left(\frac{\beta}{3}\right)
+ \frac{1}{L-1}\ln\left( \frac{\varepsilon}{F(x^{(1)}) - \eta^1}\right)\right] + L.
\end{equation}
Let us now derive an upper bound on the number $L$ of descent steps. By virtue of \eqref{eqn: n_v}
and \eqref{eqn: n_new_rule1}, descent steps are made only if
\[
F(x^k) - F(x^*) \ge \beta\varepsilon;
\]
otherwise, the method must stop. To explain it more specifically, if $F(x^k) - F(x^*) \le \beta\varepsilon$, then $F(x^k) - F(z^k_{j_k}) \le \beta\varepsilon$. If a descent step were made, $F(z^k_{j_k}) \le F(x^k) - \beta v_k$. Then $\beta v_k \le \beta\varepsilon $. Since $v_k \le \varepsilon$, the algorithm would have already stopped, which contradicts our assumption.
 It follows from Lemma \ref{lem: conv_rate_2}, that
\[
(1 - \alpha \beta)^{L-1} \big(F(x^{(1)}) - F(x^{*})\big) \ge \beta\varepsilon.
\]
Therefore,
\begin{equation}
\label{decent step bound}
L \leq 1 + \frac{\ln(\beta\varepsilon) - \ln\big(F(x^1) - F(x^{*})\big)}{\ln( 1 - \alpha \beta)} .
\end{equation}
 As a result, we have the final bound for the total number of descent and null steps:
\begin{equation}\label{total step bound}
\begin{split}
{L+\sum_{\ell=1}^L n_\ell}
 &\le
\frac{1}{\varepsilon C\ln( 1 - \alpha \beta)}\ln\frac{F(x^{(1)}) - F(x^{*})}{\beta\varepsilon}\Bigg[\ln\left(\frac{2\varphi(\alpha)}{1 + 2\varphi(\alpha)}\right)+ \ln\left(\frac{\beta}{3}\right)\Bigg] \\
&{\quad} + \frac{1}{\varepsilon C}\ln\left( \frac{F(x^{(1)}) - \eta^1}{\varepsilon}\right)
+ 2\frac{\ln(\beta\varepsilon) - \ln\big(F(x^{(1)}) - F(x^{*})\big)}{\ln( 1 - \alpha \beta)} + 2.
\end{split}
\end{equation}
Therefore, in order to achieve precision $\varepsilon$, the number of steps needed is of order
\[
L+\sum_{\ell=1}^L n_\ell \sim \mathcal{O}\Bigg( \frac{1}{\varepsilon}\ln\bigg(\frac{1}{\varepsilon}\bigg)\Bigg).
 \]
This is almost equivalent to saying that given the number of iterations $k$, the precision of the solution is approximately $\mathcal{O}(1/k)$.

\section{Application to structured regularized regression problems}
\label{s:application}
In many areas in data mining and machine learning, such as computer vision and compressed sensing, the resulting optimization models consist of a convex loss function and multiple convex regularization functions, called the \emph{composite prior models} in \cite{huang2011composite}. For example, in compressed sensing, the linear combination of the total variation (TV) norm and $L_1$ norm is a popular regularizer in recovering Magnetic Resonance (MR) images. Formally, the models are formulated as follows:
\begin{equation}\label{prob: composite_prior}
\min_{x\in \mbR^n} F(x) = f(x) + \sum_{i=1}^N h_i(B_i x),
\end{equation}
where $f$ is the loss function to measure the goodness-of-fit of the data, while the functions $h_i$ are regularization terms. All the functions are convex but not necessarily smooth.

The SLIN algorithm introduced in our paper can be directly applied to solve the general problem \eqref{prob: composite_prior}. It can be further specialized to take advantage of additional
  features of the functions involved.
  In the following subsection  we discuss one such specialization.

\subsection{Fused lasso regularization problem}
The problem is defined as follows:
\begin{equation}
\label{example-fused}
\min_x \frac{1}{2}\big\|b-Ax\big\|_2^2 + \lambda_1\big\|x\big\|_1
+ \lambda_2 \sum_{j=1}^{p-1} \big|x_{j+1} - x_j\big|,
\end{equation}
where $A$ is an $m\times n$ matrix, and $\lambda_1, \lambda_2 > 0$ are fixed parameters. This model contains two regularization terms:  the \emph{lasso} penalty $h_1(x) = \lambda_1 \|x\|_1$,
 and the  \emph{fused lasso} penalty $h_2(x) = \lambda_2 \sum_j |x_{j+1} - x_j|$. We name the
first function as $f(x) := \frac{1}{2}\big\|b-Ax\big\|_2^2$. In models with a quadratic loss function,
 we found it convenient to use the matrix $D = \text{diag}(A^T A)$ in the proximal term of the method
 \cite[Sec. 3]{lin2014alternating}.

In order to solve each subproblem, we need the gradient of $f(\cdot)$ and subgradients of the regularization functions, which are readily available.
Our method requires their explicit calculation at the initial iteration only;
at later iterations they are obtained implicitly, as described in Step 1 of the algorithm.

With these, we can solve each subproblem iteratively.

\textbf{The $f$-subproblem.}
Skipping the constants, the $f$-subproblem has the form:
\begin{equation}
\min_x \frac{1}{2}\big\|b-Ax\big\|_2^2 + g^T_{h_1}x + g^T_{h_2}x + \frac{1}{2}\big\|x-x^k\big\|^2_D.
\end{equation}
This is a unconstrained quadratic optimization problem and its optimal solution can be obtained by solving the following linear system of equations:
 \[
 (A^TA + D)x = A^Tb - g_{h_1} - g_{h_2} + D^Tx^k.
 \]
 It can be very efficiently solved by the preconditioned conjugate gradient method with preconditioner $D$, as discussed in \cite[Sec. 3]{lin2014alternating}, because the condition index of the system
 is uniformly bounded. Only matrix--vector multiplications are involved, facilitating the use of a sparse structure of $A$.
 After the solution is obtained,
the gradient of $f(x)$ and its linearization can be determined by Step 1 of the SLIN algorithm.

\textbf{The $h_1$-subproblem.}
The subproblem is defined as follows (ignoring the constants):
\begin{equation}
\min_x g^T_{f}x + \lambda_1\big\|x\big\|_1 + g^T_{h_2}x + \frac{1}{2}\big\|x-x^k\big\|^2_D.
\end{equation}
This problem is separable in the decision variables, with the following closed-form solution:
\[
(x_{h_1})_i = {\rm \text{sgn}}(\tau_i)\max\Big(0, |\tau_i| - \frac{\lambda_1}{d_i}\Big),\quad i = 1,\ldots,n.
\]
Here $\tau_i = x^k_i - \frac{(g_{f})_i + (g_{h_2})_i}{d_i}$.

The solution of the $h_1$-subproblem gives a new subgradient of $h_1$ at the minimal point.

\textbf{The $h_2$-subproblem.}
The subproblem  is defined as follows (ignoring the constants):
\[
\min_x g^T_f x + g^T_{h_1} x + \lambda_2 \sum_{j=1}^{p-1} \big|x_{j+1} - x_j\big| + \frac{1}{2} \big\|x - x^k\big\|^2_D.
\]
Exactly as described in \cite{lin2014alternating}, this problem can be equivalently formulated as a
constrained optimization problem:
\[
\min_{x,z} g^T_f x + g^T_{h_1} x + \lambda_2 \big\|z\big\|_{1} + \frac{1}{2}\big\|x - x^k\big\|^2_D,\ \text{subject to}\ Rx = z,
\]
with an $(n-1)\times n$ matrix $R$ representing the system $z_j=x_{j+1}-x_j$, $j=1,\dots,n-1$.
The Lagrangian of problem (19) has the form
\[
L(x,z,\mu) = g^T_f x + g^T_{h_1} x + \lambda_2 \big\|z\big\|_{1} + \mu^T (Rx - z) + \frac{1}{2}\big\|x - x^k\big\|^2_D,
\]
where $\mu$ is the dual variable. The minimum of the Lagrangian with respect to $z$ is finite if and only if $\big\|\mu\big\|_{\infty} \leq \lambda_2$. Under this condition, the minimum value of the $z$-terms is zero and we can eliminate
them from the Lagrangian. We arrive to its reduced form,
\[
\hat{L}(x,\mu) = g^T_f x + g^T_{h_1} x + \mu^T Rx + \frac{1}{2}\big\|x - x^k\big\|^2_D.
\]
To calculate the dual function, we minimize $\hat{L}(x,\mu)$ with respect to $x$. After elementary calculations, we obtain the solution
\[
\tilde{x}_{h_2} = x^k - D^{-1} ( g_f + g_{h_1} + R^T \mu).
\]
Substituting it back to the Lagrangian, we obtain the following dual problem:
\[
\max_{\mu} -\frac{1}{2} \mu^T RD^{-1}R^T \mu + \mu^T R(x^k - D^{-1}g_f - D^{-1}g_h), \ \text{subject to}\  \big\|\mu\big\|_{\infty} \leq \lambda_2.
\]
This problem can be treated as a box-constrained quadratic programming problem, for which many efficient algorithms are available, for example coordinate-wise optimization
\cite[Sec. 4]{lin2014alternating}. Due to the structure of $R$, the computational effort per iteration is linear in the problem dimension.

\section{Overlapping group lasso problem}
We consider the following problem
\begin{equation}
\label{example-over}
\min_x \frac{1}{2K\lambda} \big\|b - Ax\big\|_2^2 + \sum_{j=1}^K d_j \big\|x_{\Gc_j}\big\|_2
\end{equation}
where $A \in R^{m\times n}$. This model contains the first function as $f(x) := \frac{1}{2K\lambda}||b-Ax||_2^2$ where parameter $\lambda>0$ and number of groups $K$ are pre-specified parameters. The second part is a sum of regularization terms, each penalty function as $h_j(x) = d_j||x_{\Gc_j}||_2$ where the  weights $d_j>0$ are known parameters.  $\Gc_j \subseteq \{1,\ldots, p\}$ is the index set of a group of variables and $x_{\Gc_j}$ denotes the subvector of $x$  with coordinates in $\Gc_j$. This group regularizer has been proven useful in high-dimensional statistics with the capability of selecting meaningful groups of features. The groups could overlap as needed. As the quadratic term has a coefficient of $\frac{1}{2K\lambda}$, the diagonal matrix $D$ in the proximal term of the method is set to $D = \frac{1}{K\lambda}\text{diag}(A^T A)$.

\noindent\textbf{The $f$-subproblem.}
The $f$-subproblem has the form:
\[
\min_x \frac{1}{2K\lambda}\big\|b-Ax\big\|_2^2 +  \sum^K_{j=1} g^T_{j}x  + \frac{1}{2}\big\|x-x^k\big\|^2_D.
\]
It has the same structure as the $f$-subproblem of the general structured fused lasso example, and can be solved in the same way;
just the matrix $D$ is different.

\noindent\textbf{The $h_j$-subproblem.}
The $h_j$-subproblem is defined as follows (ignoring the constants):
\begin{equation}
\label{hjsub}
\min d_j\big\|x_{\Gc_j}\big\|_2 + \langle s, x\rangle + \frac{1}{2}\big\|x-x^k\big\|^2_D.
\end{equation}
where $s = g_{f} + \sum_{j'\neq j}g_{h_{j'}}$; with $g_f$ denoting a subgradient of the function $f$, and
$g_{h_{j'}}$ the subgradients of $h_{j'}$ used in \eqref{subproblem}. To simplify notation, from now on we write $\Gc$ for $\Gc_j$.

The decision variables that are outside of the current group $\Gc$, which we denote $x_{-\Gc}$, have the following closed-form solution:
\[
x_{-\Gc}= x^k_{-\Gc} - D^{-1}_{-\Gc}s_{-\Gc}.
\]
The variables in the current group $\Gc$ can be calculated as follows. If $x_{\Gc}\neq 0$, the necessary and sufficient
optimality condition for \eqref{hjsub} is the following equation:
\begin{equation}\label{firstoverlap}
\frac{d_j x_{\Gc}}{\big\|x_{\Gc}\big\|_2}+ s_{\Gc} + D_{\Gc}(x_{\Gc} - x^k_{\Gc}) = 0.
\end{equation}
We denote
\begin{equation}\label{eqn: k}
\frac{d_j}{\big\|x_{\Gc}\big\|_2} = \kappa,
\end{equation}
This leads to
\begin{equation}\label{eqn: xi}
x_i = \frac{D_{ii}x^k_i - (s)_i}{\kappa + D_{ii}},\quad i \in \Gc.
\end{equation}
Substituting into \eqref{eqn: k}, after simple manipulations, we obtain the following equation for $\kappa$:
\begin{equation}\label{eqn: k_solve}
\sum_{i\in \Gc_j}\biggl(\frac{D_{ii}x^k_i - (s)_i}{1 + \frac{ D_{ii}}{\kappa}}\biggr)^2 = d^2_j.
\end{equation}
Since the left hand side of this equation is an increasing function of $\kappa$, we can easily solve it by bisection,
if a solution exists. If the columns if $A$ are normalized, then all $D_{ii}=1$, and equation
\eqref{eqn: k_solve} can be solved in closed form.

Letting $\kappa\to \infty$ on the left hand side, we obtain the condition for the existence
of a solution of \eqref{eqn: k_solve}:
\begin{equation}
\label{existence}
\sum_{i\in \Gc_j}\big(D_{ii}x^k_i - (s)_i\big)^2 > d^2_j.
\end{equation}
If inequality \eqref{existence} is satisfied, $\kappa$ can be found by bisection and $x_{\Gc}$ follows from
  \eqref{eqn: xi}. If \eqref{existence} is not satisfied,
the only possibility is that the optimal solution of \eqref{hjsub} is $x_{\Gc}=0$.

\section{Numerical Results}
In this section, we  present some experimental results for problems
\eqref{example-fused} and \eqref{example-over}. All these studies are performed on an 1.8 GHZ, 4GB RAM computer using MATLAB.

\subsection{Fused lasso experiments}
In Tables \ref{table: running_time1} and \ref{table: running_time2}, we evaluate SLIN against {three} its modifications, in order to assess the usefulness of the main features of the method: the dynamic selection of the block to be optimized, and the sufficient improvement test. The first modification processes the blocks in a fixed order and performs the improvement test after every block, to decide whether to change the current value of $x^k$. In a way, it is a direct extension of the alternating linearization method (ALIN) of \cite{lin2014alternating}, and is labeled as such in the tables. {The second modification processes the blocks in a fixed order and updates $x^k$ after each cycle of 3 blocks.
 In the case of two blocks,  it would correspond to the Douglas--Rachford operator splitting method of \cite{Lions-Mercier}; see  \cite{lin2014alternating}. We use the name Douglas--Rachford in the tables.}
{The third modification processes the blocks in a fixed order and carries out the update of $x^k$ after each block. In the case of two blocks, it would correspond to the Peaceman--Rachford operator splitting method of \cite{Lions-Mercier},
 as explained in \cite{lin2014alternating}, and we use this name in the tables}.

 We report the average performance of all 4 versions in cases when $m > n$ $ (m=1000, n = 300)$ and when $m < n$ $ (m = 300, n = 1000)$, with different tolerance and regularization parameters. We run the experiments 10 times with different samples of the  matrix $A$ and report the average results and their standard deviation.  In the tables, ``Iterations'' denotes the average iteration number with 1000 as the default maximum iteration number; ``Time'' denotes the average CPU run time in seconds; ``Relative Error'' is defined as the relative difference between the optimal value (obtained by MATLAB ``fminunc'' function) and those obtained by the SLIN, ALIN, and operator splitting methods, respectively. ``MAXITER'' and ``NA'' indicate that the algorithm did not converge in the pre-specified number of 1000 iterations. We observe that to obtain the same accuracy, SLIN requires significantly fewer iterations and less CPU time than the other versions. The versions without sufficient progress test (na\"{i}ve extensions of operator splitting methods) did not converge in many cases. The version with the fixed block order (na\"{i}ve extension of the alternating linearization method), although convergent, was always significantly slower.

In Table \ref{table: running_time3}, we report performance of SLIN with different values of the parameter $\beta$. Based on that, we  use $\beta = 0.5$ in all our further experiments.

\begin{table}[h]
\centering
\begin{tabular}{ccccc}
  \toprule[0.1em]
 \multirow{2}{*}{Parameters} & \multirow{2}{*}{Method} & \multirow{2}{*}{Iterations}  & \multirow{2}{*}{Time} & \multirow{2}{*}{Relative Error}\\
  \\
  \toprule[0.1em]
  $tol = 10^{-3}$ & Selective Linearization &55(28.75) &  0.99(0.47)& 0.0156(0.0043)\\
  $\lambda = 1\tau$ & Alternating Linearization &66(31.04) & 1.10 (0.99) & 0.0161(0.0048) \\
    & Douglas-Rachford & 323(78.57) & 5.53 (0.96)& 0.0161(0.0046)\\
  & Peaceman-Rachford & MAXITER(0.00) & 16.93 (0.23)& 0.9191(0.2123)\\
    \toprule[0.1em]
  $tol = 10^{-7}$ &  Selective Linearization & 46 (25.79) &  1.00 (0.41) & 0.0154(0.0075) \\
  $\lambda = 1\tau$ & Alternating Linearization & 173 (31.30) &  2.98 (0.60)&
 0.0166(0.0078) \\
  & Douglas-Rachford & 246(42.57) & 4.32 (0.73)& 0.0155(0.0076)\\
  & Peaceman-Rachford & MAXITER (0.00) &  16.75 (0.22) & NA(NA)\\
   \toprule[0.1em]
   $tol = 10^{-3}$ &  Selective Linearization & 79(2.30) &  1.60 (0.07) & 0.0004(0.0000)\\
  $\lambda = 10^{-2}\tau$ &  Alternating Linearization & 91 (3.05) &  1.64 (0.06) &0.0004(0.0000)\\
  & Douglas-Rachford & 205(10.17) & 3.45 (0.21)& 0.0004(0.0000)\\
 & Peaceman-Rachford & MAXITER (0.00) &  16.57 (0.11) & NA(NA)\\
    \toprule[0.1em]
   $tol = 10^{-7}$ &   Selective Linearization & 177(20.68) &  3.06(0.88) & 0.0028(0.0019)\\
   $\lambda = 10^{-2}\tau$ & Alternating Linearization & 450(50.74)&  8.01(4.06) & 0.0028(0.0019)\\
   & Douglas-Rachford & 573(59.81) & 9.91 (0.52)& 0.0028(0.0019)\\
  &Peaceman-Rachford & MAXITER (0.00) &  16.60(0.15) & NA(NA)\\
  \bottomrule[0.1em]
\end{tabular}
\caption{Comparison of methods on problems  with $m = 1000$ and $n = 300$.}\label{table: running_time1}
\end{table}

\begin{table}[h]
\centering
\begin{tabular}{ccccc}
  \toprule[0.1em]
 \multirow{2}{*}{Parameters} & \multirow{2}{*}{Method} & \multirow{2}{*}{Iterations}  & \multirow{2}{*}{Time} & \multirow{2}{*}{Relative Error}\\
  \\
  \toprule[0.1em]
    $tol = 10^{-3}$ &Selective Linearization & 83(14.15) & 11.26(3.01) & 0.0393(0.0103)\\
   $\lambda = 1\tau$ & Alternating Linearization & 425(13.96) &  55.97(4.87) & 0.0456(0.0126)\\
   & Douglas-Rachford & MAXITER(0.00) & 130.39 (6.42)& 0.0355(0.0096)\\
 & Peaceman-Rachford & MAXITER(0.00) &  131.19 (8.48)& NA(NA)\\
    \toprule[0.1em]
    $tol = 10^{-7}$ &  Selective Linearization & 257 (7.97) &  36.76(13.93) & 0.0165(0.0059)\\
   $\lambda = 1\tau$ &  Alternating Linearization & 419(8.23) & 59.86(1.99) & 0.0240(0.0083) \\
      & Douglas-Rachford & MAXITER(0.00) & 144.31 (18.14)& 0.0187(0.0059)\\
&  Peaceman-Rachford & MAXITER(0.00) &  143.93 (11.56)& 0.0118(0.0063)\\
   \toprule[0.1em]
    $tol = 10^{-3}$ &  Selective Linearization & 210(27.96) &  27.30(18.07) & 0.0483(0.0160) \\
    $\lambda = 10^{-2}\tau$ &  Alternating Linearization & 290 (34.73) &  39.86(28.05) & 0.0484(0.0160)\\
          & Douglas-Rachford & 291(53.85) & 39.51 (32.29)& 0.0484(0.0160)\\
 & Peaceman-Rachford & MAXITER(0.00) &139.38(18.88) & NA(NA)\\
    \toprule[0.1em]
  $tol = 10^{-7}$ &  Selective Linearization & 404 (45.45) &  49.51(8.04) & 0.0550(0.0181) \\
     $\lambda = 10^{-2}\tau$ & Alternating Linearization & 780(63.45) & 107.23(5.01) &0.0550(0.0181)\\
  & Douglas-Rachford & 891(53.07) & 106.94 (12.29)& 0.0550(0.0181)\\
 & Peaceman-Rachford & MAXITER(0.00) & 133.34(7.81) & NA(NA)\\
  \bottomrule[0.1em]
\end{tabular}
\caption{Comparison of methods on problems with $m = 300$ and $n = 1000$}\label{table: running_time2}
\end{table}

\begin{table}[h]
\centering
\begin{tabular}{cccc}
  \toprule[0.1em]
  \multirow{2}{*}{$\beta$} & \multirow{2}{*}{Iterations}  & \multirow{2}{*}{Time} & \multirow{2}{*}{Relative Error}\\
  \\
  \toprule[0.1em]
  0.2 & 76(25.61) & 0.82(0.95) &  0.0200(0.0100)\\
  0.5 & 28 (24.70)&  0.39 (0.64)& 0.0080(0.0037)\\
 0.8 & 74(33.72) & 1.08 (0.63)& 0.0400(0.0200)\\
    \bottomrule[0.1em]
\end{tabular}
\caption{The effect of different values of  $\beta$ in the SLIN algorithm for the problem with $n = 1000$ and $p = 300$.}\label{table: running_time3}
\end{table}

\subsection{Overlapping group lasso experiments}
We compared our method with existing algorithms in the tree-structured, fixed order, and random order cases. It shows that SLIN does not care how the groups are partitioned and is applicable to arbitrary group-splitting cases.
\subsubsection{Tree-structured overlapping groups}
In a tree-structured overlapping group lasso problem, described in \cite{JMOB11}, the groups correspond to nodes of a tree. Thus, for any two groups, either one is a subset of the other, or they are disjoint. The design matrix and input vector are centered and normalized to have unit $\ell_2$-norms. We conduct the speed comparisons between our approach and FISTA \cite{JMOB11}. From Table \ref{table: running_time4} we can see that the SLIN algorithm is faster in terms of both iteration number and computational time.

\begin{table}[h!]
\centering
\begin{tabular}{cccc}
  \toprule[0.1em]
  \multirow{2}{*}{Parameters} & \multirow{2}{*}{Methods} &\multirow{2}{*}{Iter}  & \multirow{2}{*}{Time} \\
  \\
  \toprule[0.1em]
$ m=100, n =10$&SLIN & 11.60 (0.70) & 0.0897(0.0035) \\
 $K =8$ &FISTA & 25.20(2.85)& 0.1385(0.0138) \\
    \bottomrule[0.1em]
\end{tabular}
\caption{Comparison of SLIN and FISTA on tree-structured overlapping group lasso problem. }\label{table: running_time4}
\end{table}

\subsubsection{Fixed order overlapping groups}
We simulate data for a univariate linear regression model with an overlapping group structure. The entries are sampled from i.i.d. normal distributions, $ x_j = (-1)^j \exp(-(j-1)/100)$, and $b = Ax + \varepsilon$, with the noise $\varepsilon$  sampled from the standard normal distribution.  Assuming that the inputs are ordered, we define a sequence of $K$ groups of 100 adjacent inputs with an overlap of 10 variables between two successive groups, so that
\begin{equation}\label{eqn: structure_group}
G = \{\{1,\ldots,100\},\{91,\ldots,190\},\ldots,\{n-99,\ldots,p\}\},
\end{equation}
where $n = 90K + 10$.  We adopt uniform weights $ d_j = 1/K$ and set $\lambda = K/5$.

To demonstrate the efficiency and scalability of the SLIN algorithm, we compared SLIN with several specialized methods
for overlapping group lasso problems: PDMM of \cite{chatzipanagiotis2015augmented,WBL14}, sADMM or Jacobian ADMM of \cite{DLY13}, PA-APG of \cite{YY12} and S-APG of \cite{CLKCX12}.  All experiments were run sequentially, that is, no parallel processing features were exploited. We run the experiments 10 times with different samples of the  matrix
$A$; we report the average results.

Figure \ref{figure: compare_line} plots the convergence of the objective function values \emph{versus} the number of iterations,
for the number of groups $K = 100$.  For the dual methods PDMM and sADMM, we report the values of the {augmented Lagrangian}. They go from super optimal (because the iterates are infeasible) and converge to the optimal value. The SLIN algorithm is the fastest, in terms of iterations while PDMM  is the second, and sADMM the third. The two accelerated methods, PA-APG  and S-APG, are the slowest in these tests.

\begin{figure}[h!]
\centering
\includegraphics[scale=0.7]{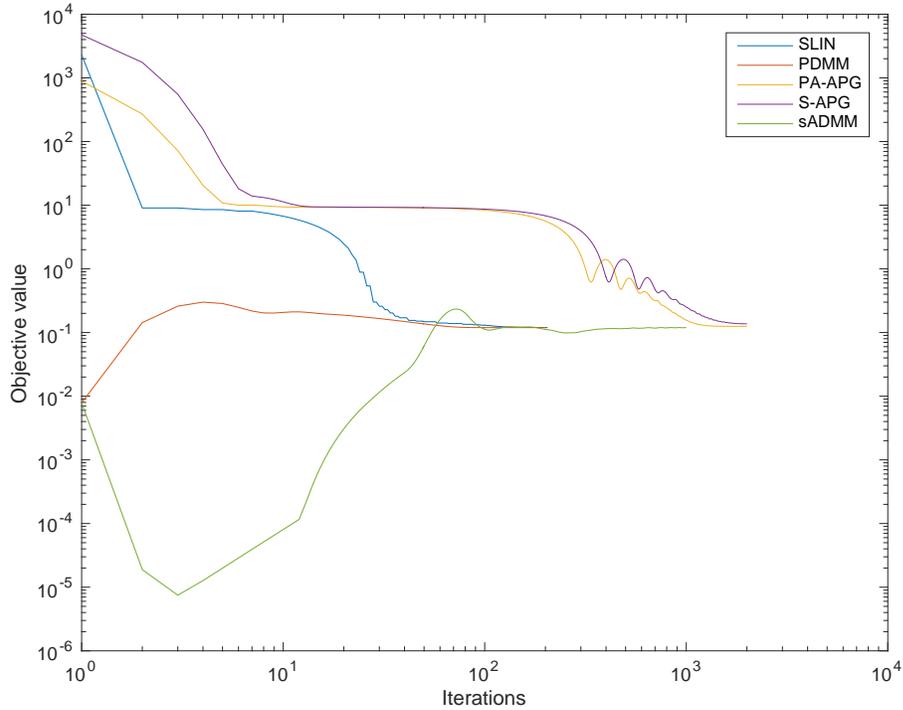}
\caption{Comparison of  SLIN  and other algorithms on an overlapping group lasso problem.}\label{figure: compare_line}
\end{figure}
\subsubsection{Randomly overlapping groups}
 In the next stage,  we conducted additional comparisons between SLIN and PDMM on group lasso problems with randomized overlapping, which do not exhibit the regular group structure specified in (\ref{eqn: structure_group}).

This type of problem arises in applications such as bioinformatics [3], where one uses prior information to model potential overlapping of groups of variables. For example, in high throughput gene expression analysis, the number of parameters to be estimated is much greater than the sample size. One often utilizes information including gene ontology to define group overlaps among genes, thereby achieving structured regularization \cite{VRMV14}. The resulting overlaps are ``arbitrary'' (depending on the specific gene ontology) and more complex than the systematic overlapping example described in (\ref{eqn: structure_group}). We generated test cases in which the indices in each of the 100 groups were assigned to the $n$ locations. As a result, the number of overlapping variables between the groups was random, and multiple group membership was possible. The performance of the two methods on randomized overlapping group lasso problems is summarized in Tables \ref{table: running_time4}  and \ref{table: running_time5}.

For fair comparison of the methods, we run PDMM on each instance of the problem. PDMM was set to run to $\text{``tol''} = 10^{-4}$ or 2,000 iterations, whichever came first. We set the tuning parameters $d_g = 0.01/K$, and $0.02/K$, respectively. Then SLIN was set to run until the objective function values obtained were as good as that of PDMM.  We run the experiments 10 times with different samples of the randomly generated groups; in Tables \ref{table: running_time5} and \ref{table: running_time6}  we report the average results and their standard deviations. In all cases, the number of iterations of SLIN is much smaller than that of PDMM.
In the determined cases, where $m=1000$ and $n=800$, the running time of SLIN is usually better than
that of PDMM. In the under-determined cases, where $m=500$ and $n=600$, the running time of SLIN is slightly worse than that of PDMM.


In summary, we can conclude that SLIN is a highly efficient and reliable general-purpose method for multi-block optimization of convex nonsmooth functions. It successfully competes with dedicated methods for special classes of problems.

\begin{table}[h!]
\centering
\begin{tabular}{cccc}
  \toprule[0.1em]
  \multirow{2}{*}{Parameters} & \multirow{2}{*}{Methods} &\multirow{2}{*}{Iter}  & \multirow{2}{*}{Time} \\
  \\
  \toprule[0.1em]
$ m=1000, n =800$&SLIN & 280 (15.82) & 5.42(0.36) \\
 $K =80, d_g = 0.01/K$ &PDMM & 638(21.14)& 5.48(0.10)  \\
       \toprule[0.1em]
$ m=1000, n =800$&SLIN & 308 (16.11)& 5.38(0.43) \\
$ K =90, d_g = 0.01/K$ &PDMM & 836(29.33) & 7.19(0.39) \\
      \toprule[0.1em]
$ m=1000, n =800$&SLIN & 331 (13.05)& 5.89(0.28) \\
$ K =100, d_g = 0.01/K$ &PDMM & 991 (75.51) & 8.91(0.81) \\
   \toprule[0.1em]
$ m=1000, n =800$&SLIN & 306 (11.79)& 5.53(0.23) \\
$ K =80, d_g = 0.02/K$ &PDMM & 560(59.47) & 4.75(0.45) \\
     \toprule[0.1em]
$ m=1000, n =800$&SLIN & 330 (11.19)& 5.63(0.22)  \\
 $K =90, d_g = 0.02/K$ &PDMM & 620(38.33) & 5.65(0.91) \\
    \toprule[0.1em]
      $ m=1000, n =800$&SLIN & 364 (13.77)& 6.09(0.31)  \\
 $K =100, d_g = 0.02/K$ &PDMM & 741(58.06) & 4.14 (0.47) \\
    \bottomrule[0.1em]
\end{tabular}
\caption{Comparison SLIN and PDMM in solving an overlapping group lasso of randomly generated groups. Determined cases with $m=1000$ and $n=800$.}\label{table: running_time5}
\end{table}
\begin{table}[h!]
\centering
\begin{tabular}{cccc}
  \toprule[0.1em]
  \multirow{2}{*}{Parameters} & \multirow{2}{*}{Methods} &\multirow{2}{*}{Iter}  & \multirow{2}{*}{Time} \\
  \\
  \toprule[0.1em]
$ m=500, n =600$&SLIN & 1280 (39.73) & 17.75(0.60) \\
 $K =80, d_g = 0.01/K$ &PDMM & 1453(157.07)& 12.01(1.03)  \\
       \toprule[0.1em]
$ m=500, n =600$&SLIN & 1119 (57.76)& 16.61(1.00) \\
$ K =90, d_g = 0.01/K$ &PDMM & 1566(61.63) & 13.88(1.63) \\
      \toprule[0.1em]
$ m=500, n =600$&SLIN & 973 (39.46)& 13.72(1.241) \\
$ K =100, d_g = 0.01/K$ &PDMM & 1753 (122.27) & 16.75(2.64) \\
   \toprule[0.1em]
$ m=500, n =600$&SLIN & 792 (36.58)& 9.79(0.73) \\
$ K =80, d_g = 0.02/K$ &PDMM & 968(39.82) & 7.53(0.54) \\
     \toprule[0.1em]
$ m=500, n =600$&SLIN & 722 (23.46)& 8.87(0.65)  \\
 $K =90, d_g = 0.02/K$ &PDMM & 1170(102.16) & 9.57(1.00) \\
    \toprule[0.1em]
      $ m=500, n =600$&SLIN & 683 (15.66)& 8.01(0.41)  \\
 $K =100, d_g = 0.02/K$ &PDMM & 1208(70.30) & 10.81 (1.21) \\
    \bottomrule[0.1em]
\end{tabular}
\caption{Comparison SLIN and PDMM in solving an overlapping group lasso of randomly generated groups. Underdetermined cases with $m=500$ and $n=600$}\label{table: running_time6}
\end{table}


\newcommand{\etalchar}[1]{$^{#1}$}

\end{document}